\documentclass[11pt]{amsart}%
\usepackage{amsmath}
\usepackage{graphicx}
\usepackage{amsfonts}
\usepackage{amssymb}%
\newtheorem{theorem}{Theorem}[section]
\theoremstyle{plain}

\newtheorem{proposition}{Proposition}[section]

\numberwithin{equation}{section}
\begin{document}
\title[Singular solutions]{Singular solutions to special Lagrangian equations with subcritical phases and
minimal surface systems}
\author{Dake WANG}
\author{Yu YUAN}
\address{Department of Mathematics, Box 354350\\
University of Washington\\
Seattle, WA 98195}
\email{dkpool@math.washington.edu, yuan@math.washington.edu}
\thanks{Both authors are partially supported by an NSF grant.}
\date{\today}

\begin{abstract}
We construct singular solutions to special Lagrangian equations with
subcritical phases and minimal surface systems. A priori estimate breaking
families of smooth solutions are also produced correspondingly. A priori
estimates for special Lagrangian equations with certain convexity are largely
known by now.

\end{abstract}
\maketitle

\section{\bigskip Introduction}

In this paper, we construct \emph{singular} solutions to the special
Lagrangian equation%

\begin{equation}
\sum_{i=1}^{n}\arctan\lambda_{i}=\Theta\label{EsLag}%
\end{equation}
with subcritical phase $\left\vert \Theta\right\vert <\left(  n-2\right)
\pi/2,$ where $\lambda_{i}$ are the eigenvalues of $D^{2}u,$ and the minimal
surface system for $k$-vector valued functions of $n$-variables
\begin{equation}
\bigtriangleup_{g}U=\sum_{i,j=1}^{n}\frac{1}{\sqrt{g}}\partial_{x_{i}}\left(
\sqrt{g}g^{ij}\partial_{x_{i}}U\right)  =0, \label{MSS}%
\end{equation}
where the induced metric
\[
g=I+\left(  DU\right)  ^{T}DU.
\]
Equation (\ref{EsLag}) is the potential equation for (\ref{MSS}) with
solutions $U=Du.$ The Lagrangian graph $\left(  x,Du\left(  x\right)  \right)
\subset\mathbb{R}^{n}\times\mathbb{R}^{n}$ is called special and in fact
volume minimizing when the phase or the argument of the complex number
$\left(  1+\sqrt{-1}\lambda_{1}\right)  \cdots\left(  1+\sqrt{-1}\lambda
_{n}\right)  $ is constant $\Theta,$ or equivalently $u$ satisfies equation
(\ref{EsLag}); see the work [HL1, Theorem 2.3, Proposition 2.17] by Harvey and
Lawson. The phase $\left(  n-2\right)  \pi/2$ is said critical because the
level set $\left\{  \left.  \lambda\in\mathbb{R}^{n}\right\vert \lambda
\ \text{satisfying }(\ref{EsLag})\right\}  $ is convex \emph{only} when
$\left\vert \Theta\right\vert \geq\left(  n-2\right)  \pi/2$ [Y2, Lemma 2.1].
In dimension three, when $\left\vert \Theta\right\vert =$ $\pi/2$ or
$\left\vert \Theta\right\vert =0,\ \pi,$ equation (\ref{EsLag}) also takes the
quadratic and cubic algebraic forms respectively%
\begin{equation}
\sigma_{2}\left(  D^{2}u\right)  =\lambda_{1}\lambda_{2}+\lambda_{2}%
\lambda_{3}+\lambda_{3}\lambda_{1}=1 \label{s2}%
\end{equation}
or%
\begin{equation}
\bigtriangleup u=\det D^{2}u. \label{Lap=MA}%
\end{equation}

We state our first main result.

\begin{theorem}
There exist $C^{1,1/\left(  2m-1\right)  }$ ($m=2,3,4,\cdots$) viscosity
solutions $u^{m}$ to (\ref{EsLag}) with $n=3$ and each $\Theta\in\left(
-\frac{\pi}{2},\frac{\pi}{2}\right)  ,$ such that $u^{m}\in C^{1,1/\left(
2m-1\right)  }\left(  B_{1}\right)  \cap C^{\infty}\left(  B_{1}%
\backslash\{0\}\right)  $ for $B_{1}\subset\mathbb{R}^{3}$ but $u^{m}\notin
C^{1,\delta}$ for any $\delta>1/\left(  2m-1\right)  .$
\end{theorem}

Rotating forth and back, we obtain our second (\textquotedblleft
smooth\textquotedblright) result.

\begin{theorem}
There exist a family of smooth solutions $u^{\varepsilon}$ to (\ref{EsLag}) in
$B_{1}\subset\mathbb{R}^{3}$ with $n=3$ and each fixed $\Theta\in\left(
-\frac{\pi}{2},\frac{\pi}{2}\right)  $ such that%
\[
\left\Vert Du^{\varepsilon}\right\Vert _{L^{\infty}\left(  B_{1}\right)  }\leq
C\ \ \ \text{but\ \ }\left\vert D^{2}u^{\varepsilon}\left(  0\right)
\right\vert \rightarrow\infty\ \ \text{as }\varepsilon\rightarrow0.
\]

\end{theorem}

For each $u^{\varepsilon}$ with small $\varepsilon$ fixed in Theorem 1.2, the
Hessian $\left\vert D^{2}u^{\varepsilon}\left(  0\right)  \right\vert $ (in
the max eigenvalue norm) is strictly larger that its nearby values in the
three dimensional domain of the solution to a now uniformly elliptic equation
(\ref{EsLag}). (It can be seen by Property 2.4 in Section 2 and tracing the
eigenvalue dependency in Section 4.) This violates the maximum principle. In
contrast to the two dimensional fully nonlinear uniformly elliptic equations,
it is classically known that the Hessian of any solution enjoys the maximum
principle. To the solutions in the above two theorems, by adding quadratics of
extra variables in higher dimensions $n\geq4,$ we immediately get the
corresponding counterexamples for (\ref{EsLag}) with all subcritical phases
$\left\vert \Theta\right\vert <\left(  n-2\right)  \pi/2.$ Furthermore, we
convert our counterexamples to the ones for minimal surface system (\ref{MSS}).

\begin{theorem}
There exist a family of weak solutions $U^{m}$ to (\ref{MSS}) in $B_{1}%
\subset\mathbb{R}^{3}$ with $n=3$, $k=3,$ and $m=2,3,4,\cdots$ such that%
\[
U^{m}\in W^{1,p}\left(  B_{1}\right)  \ \text{\ for any }p<\frac{2m+1}%
{2m-2}\ \ \text{but\ \ }U^{m}\notin W^{1,\frac{2m+1}{2m-2}}\left(
B_{1}\right)  .
\]
Furthermore, there exist a family of smooth solutions $U^{\varepsilon}$ to
(\ref{MSS}) in $B_{1}\subset\mathbb{R}^{3}$ with $n=3$ and $k=3$ such that%
\[
\left\Vert U^{\varepsilon}\right\Vert _{L^{\infty}\left(  B_{1}\right)  }\leq
C\ \ \ \text{but\ \ }\left\vert DU^{\varepsilon}\left(  0\right)  \right\vert
\rightarrow\infty\ \ \text{as }\varepsilon\rightarrow0.
\]

\end{theorem}

The vector valued functions $U^{m}$ are taken as $Du^{m}$ with $u^{m}$ from
Theorem 1.1, thus the first part of the theorem gives a negative answer to
Nadirashvili's question whether there is an $\varepsilon$ improvement of
$W^{2,1}$ solutions to special Lagrangian equation (\ref{EsLag}) in general.
We are grateful for this question. In terms of minimal surface system
(\ref{MSS}), the question would be whether there is an $\varepsilon$
improvement of $W^{1,1}$ solutions.

For special Lagrangian equation (\ref{EsLag}) with critical and supercritical
phases $\left\vert \Theta\right\vert \geq\left(  n-2\right)  \pi/2$ in
dimension two and three, with very large phase $\left\vert \Theta\right\vert
\geq\left(  n-1\right)  \pi/2$ in general dimensions, a priori Hessian and
gradient estimates, and consequently, armed with the solvability of the
Dirichlet problem with smooth boundary data to the now convex special
Lagrangian equation (\ref{EsLag}) in the critical and supercritical phase
cases, the regularity of $C^{0}$ viscosity solutions were derived in [WY1]
[WY2] [WY3] [CWY]. In passing, we also mention that the existence and
uniqueness of the $C^{0}$ viscosity solution for the Dirichlet problem to
strictly elliptic equation (\ref{EsLag}) is known (cf. [CWY, p. 594]). In
recent years, there has been a new approach toward the existence and
uniqueness of $C^{0}$ viscosity or weak solution for the Dirichlet problem to
strictly elliptic as well as degenerate elliptic fully nonlinear equations by
Harvey and Lawson [HL2] [HL3].

Recently Nadirashvili and Vl\u{a}duc\c{t} [NV] constructed beautiful
$C^{1,1/3}$ singular viscosity solutions to (\ref{EsLag}) with subcritical
phases $\left\vert \Theta\right\vert <\pi/2$ in dimension three, relying on
\textquotedblleft brutal force\textquotedblright\ calculations (for the
approximate solutions) and a hard and deep topological result in [EL] (for the
injectivity of the gradient maps).

For minimal surface equations, namely (\ref{MSS}) with $k=1,$ the gradient
estimate in terms of the height of the minimal surfaces, is the classic result
by Bombieri-De Giorgi-Miranda [BDM], from which it follows the regularity of
weak or viscosity solutions. For smooth solutions to (\ref{MSS}) with $n=2,$
Gregori [G] extended Heinz's Jacobian estimate to get a gradient bound in
terms of the heights of the two dimensional minimal surfaces with any
codimension. For smooth solutions to general minimal surface system
(\ref{MSS}) with certain constraints on the gradients themselves, a gradient
estimate was obtained by Wang [W], using an integral method developed for
codimension one minimal graphs. Nonetheless, there do exist singular
$W^{1,2-}$ weak solutions to (\ref{MSS}) with $n=2;$ see Osserman [O]. Now
gradient estimates for (\ref{MSS}) with $k=2$ and $n\geq3$ still remain
mysterious and challenging.

Our construction goes as follows. In the first stage, we solve the special
Lagrangian equation (\ref{EsLag}) with the critical phase by
Cauchy-Kowalevskaya. The approximate solutions or initial data for the
relatively \textquotedblleft easier\textquotedblright\ corresponding quadratic
equation (\ref{s2}) are built up via a \emph{systematic} procedure, which
allows us to have the approximation at arbitrarily high order (Property 2.1
and 2.2), and eventually those highly (\textquotedblleft
oddly\textquotedblright\ $C^{1,1/\left(  2m-1\right)  }$) singular solutions
in Theorem 1.1 and Theorem 1.3. In the second stage, we take an
\textquotedblleft inversion\textquotedblright\ $\frac{\pi}{2}$ rotation of the
solutions from the first stage to obtain those singular solutions with phase
$0$ (Proposition 3.1). The singular solutions with other subcritical phases
are achieved via a preliminary \textquotedblleft horizontal\textquotedblright%
\ rotation before the \textquotedblleft inversion\textquotedblright%
\ $\frac{\pi}{2}$ rotation (Step 1 of Section 3). Some remarks are in order.
Those $U\left(  n\right)  $ rotations are \textquotedblleft
obvious\textquotedblright\ to produce for the $U\left(  n\right)  $ invariant
special Lagrangian equation (\ref{EsLag}). But it is by no means easy to
justify that the special Lagrangian submanifold is still a graph in the
rotated new coordinate system, thus a valid equation (\ref{EsLag}) to work on.
(Earlier development of those $U\left(  n\right)  $ rotations for
(\ref{EsLag}) can be found in [Y2] [Y3].) Here our \emph{elementary analytic}
justification for the \textquotedblleft inversion\textquotedblright%
\ $\frac{\pi}{2}$ rotation (Proposition 3.1) avoids a hard and deep
topological formula of [EL], which was employed in [NV]. Lastly we point out
that the  Legendre transformation (usually used for convex functions),  is
just the \textquotedblleft inversion\textquotedblright\ $\frac{\pi}{2}$
rotation followed by a conjugation for converting \textquotedblleft
gradient\textquotedblright\ graph $\left(  x,Du\left(  x\right)  \right)  $ to
the one $\left(  Du^{\ast}\left(  y\right)  ,y\right)  $ (now with saddle
potentials $u$ and $u^{\ast}$).  In the third stage, we kick in a little bit
extra to the preliminary \textquotedblleft horizontal\textquotedblright%
\ rotations of Stage 2, then after the same \textquotedblleft
inversion\textquotedblright\ $\frac{\pi}{2}$ rotation, we make up a
corresponding little bit \textquotedblleft backward\textquotedblright%
\ rotation to finally generate the desired family of smooth solutions in
Theorem 1.2, which break a priori Hessian estimates for special Lagrangian
equation (\ref{EsLag}) with subcritical phase. Note that here one cannot
produce those a priori estimate breaking family of smooth solutions by the
usual way, that is to solve the Dirichlet problem with smooth approximate
boundary data of the boundary value of a singular solution, as Theorem 1.1
shows the \emph{non-solvability} of smooth solution to the Dirichlet problem
to (\ref{EsLag}) of subcritical phase even with smooth boundary data. The
Dirichlet problem to the saddle branch of (\ref{Lap=MA}) or the equivalent
(\ref{EsLag}) with $n=3$ and $\Theta=0$ was \textquotedblleft
invited\textquotedblright\ by Caffarelli, Nirenberg, and Spruck in [CNS].

In closing, we point out that any further regularity beyond continuity for
continuous viscosity solutions to general special Lagrangian equation
(\ref{EsLag}) is unknown. We are also curious to know whether there exist
other $C^{1,\alpha}$ (no better) singular solutions to (\ref{EsLag}) with, in
particular, irrational exponents $\alpha$ between those odd reciprocals
$1/\left(  2m-1\right)  .$ Meanwhile, we guess that all $C^{1,\alpha}$ for
$\alpha>\frac{1}{3}$ solutions to special\ Lagrangian equation (\ref{EsLag})
with $n=3$ should be regular (analytic). This regularity for $C^{1,1}$
solutions to (\ref{EsLag}) in dimension three was shown in [Y1]. Earlier on,
Urbas [U, Theorem 1.1] proved the regularity for better than Pogorelov
solutions, namely all $C^{1,\alpha}$ for $\alpha>1-\frac{2}{n}$ (convex)
solutions to the (dual) Monge-Amp\`{e}re equation $\ln\det D^{2}u=\ln
\lambda_{1}+\cdots+\ln\lambda_{n}=c$ are $C^{3,\beta}$ and eventually analytic.

\section{Cauchy-Kowalevskaya with critical phase $\Theta=\frac{\pi}{2}$}

As a preparation for the constructions in the next three sections, we solve
the following special Lagrangian equation with critical phase in dimension
three by Cauchy-Kowalevskaya. The quadratic nature of the equation at the
critical phase is easier to work with than the cubic nature of the equations otherwise.

Our approximate solution $P\left(  x\right)  $ to the equation%
\begin{equation}
\left\{
\begin{array}
[c]{l}%
\sigma_{2}\left(  D^{2}u\right)  =\frac{1}{2}\left[  \left(  \bigtriangleup
u\right)  ^{2}-\left\vert D^{2}u\right\vert ^{2}\right]  =1\ \ \text{or }%
\sum_{i=1}^{3}\arctan\lambda_{i}=\frac{\pi}{2}\\
u_{3}\left(  x_{1},x_{2},0\right)  =P_{3}\left(  x_{1},x_{2},0\right) \\
u\left(  x_{1},x_{2},0\right)  =P\left(  x_{1},x_{2},0\right)
\end{array}
\right.  \label{Eck}%
\end{equation}
is a polynomial of degree $2m$%
\[
P=\frac{1}{2}\left(  x_{1}^{2}+x_{2}^{2}\right)  +\operatorname{Re}Z^{m}%
x_{3}+\frac{m^{2}}{4}\rho^{2m-2}x_{3}^{2}+\varepsilon\sum_{j=0}^{m}a_{j}%
x_{3}^{2m-2j}\rho^{2j},
\]
where $Z=x_{1}+\sqrt{-1}x_{2}=\rho\exp\left(  \sqrt{-1}\theta\right)  ,$
coefficients $\varepsilon$ and $a_{j}s$ are to be determined later. We
construct this $P$ satisfying the following four properties, so does $u$ then,
for $\left\vert x\right\vert =r\leq r_{m}$ with positive $r_{m}$ depending
only on $m.$

\noindent Property 2.1. $\sigma_{2}\left(  D^{2}P\right)  -1=\left[
r^{3m-3}\right]  ,$ here $\left[  r^{k}\right]  $ represents an analytic
function starting from order $k.$ Then the solution $u$ coincide with $P$ up
to order $3m-2\ $($\geq2m$ for $m\geq2,3,4,\cdots$).

\noindent Property 2.2. The three eigenvalues of $D^{2}P,$ then also $D^{2}u$
satisfy%
\begin{align*}
\lambda_{1}  &  =1+\left[  r^{m-1}\right] \\
\lambda_{2}  &  =1+\left[  r^{m-1}\right] \\
-\delta_{2}\left(  m\right)  r^{2m-2}  &  \leq\lambda_{3}\leq-\delta
_{1}\left(  m\right)  r^{2m-2}%
\end{align*}

\noindent Property 2.3. The \textquotedblleft gradient\textquotedblright%
\ graph%
\begin{gather*}
\left(  x,Du\right)  =\\
\left(
\begin{array}
[c]{c}%
x,x_{1}+O\left(  \rho\right)  \left[  r^{m-1}\right]  +\left[  r^{2m}\right]
,\ x_{2}+O\left(  \rho\right)  \left[  r^{m-1}\right]  +\left[  r^{2m}\right]
,\ \\
\operatorname{Re}Z^{m}+\frac{m^{2}}{2}\rho^{2m-2}x_{3}-2m\varepsilon
x_{3}^{2m-1}+\varepsilon\rho^{2}\left[  r^{2m-3}\right]  +\left[
r^{2m}\right]
\end{array}
\right)  .
\end{gather*}
\noindent Property 2.4. The gradient $Du$ satisfies%
\[
\delta_{3}\left(  m\right)  r^{2m-1}\leq\left\vert Du\left(  x\right)
\right\vert \leq\delta_{4}\left(  m\right)  r.
\]

We first find the equation near a quadratic solution. Let%
\[
u=\frac{1}{2}\left(  \mu_{1}x_{1}^{2}+\mu_{2}x_{2}^{2}+\mu_{3}x_{3}%
^{2}\right)  +w\left(  x\right)  .
\]
Then%
\begin{gather*}
\sigma_{2}\left(  D^{2}u\right)  -1=\frac{1}{2}\left[  \left(  \bigtriangleup
u\right)  ^{2}-\left\vert D^{2}u\right\vert ^{2}\right]  -1\\
=\frac{1}{2}\left[  \left(  \mu_{1}+\mu_{2}+\mu_{3}+\bigtriangleup w\right)
^{2}-\sum_{i=1}^{3}\left(  \mu_{i}+w_{ii}\right)  ^{2}-2w_{12}^{2}-2w_{23}%
^{2}-2w_{13}^{2}\right]  -1\\
=\mu_{1}\left(  \bigtriangleup w-w_{11}\right)  +\mu_{2}\left(  \bigtriangleup
w-w_{22}\right)  +\mu_{3}\left(  \bigtriangleup w-w_{33}\right)  +\frac{1}%
{2}\left[  \left(  \bigtriangleup w\right)  ^{2}-\left\vert D^{2}w\right\vert
^{2}\right] \\
+\mu_{1}\mu_{2}+\mu_{2}\mu_{3}+\mu_{3}\mu_{1}-1.
\end{gather*}
Set $\mu_{1}=\mu_{2}=1$ and $\mu_{3}=0,$ we get%
\begin{align*}
\sigma_{2}\left(  D^{2}u\right)  -1  &  =w_{11}+w_{22}+2w_{33}+\frac{1}%
{2}\left[  \left(  \bigtriangleup w\right)  ^{2}-\left\vert D^{2}w\right\vert
^{2}\right] \\
&  =\tilde{\bigtriangleup}w+\frac{1}{2}\left[  \left(  \bigtriangleup
w\right)  ^{2}-\left\vert D^{2}w\right\vert ^{2}\right]  ,
\end{align*}
where $\tilde{\bigtriangleup}=\partial_{11}+\partial_{22}+2\partial_{33}.$ To
make the right hand side of the above equation vanish at high orders, we
choose $w=h+Q+H,$ where%
\begin{align*}
h  &  =\operatorname{Re}Z^{m}x_{3},\ \text{an ad hoc \textquotedblleft
harmonic\textquotedblright\ function;}\\
Q  &  =\frac{m^{2}}{4}\rho^{2m-2}x_{3}^{2},\ \text{to match }\sigma_{2}\left(
D^{2}h\right)  ;\\
H  &  =\varepsilon\left(  -x_{3}^{2m}+\sum_{j=1}^{m}a_{j}x_{3}^{2m-2j}%
\rho^{2j}\right)  ,\ \text{to make eigenvalue }\lambda_{3}\ \text{negative.}%
\end{align*}
Then
\[
\sigma_{2}\left(  D^{2}u\right)  -1=\underset{0}{\underbrace{\tilde
{\bigtriangleup}h}}+\tilde{\bigtriangleup}Q+\tilde{\bigtriangleup}H+\frac
{1}{2}\left[  \left(  \bigtriangleup h\right)  ^{2}-\left\vert D^{2}%
h\right\vert ^{2}\right]  +\left[  r^{3m-3}\right]  .
\]
A simple calculation leads to%
\[
D^{2}h=\left[
\begin{array}
[c]{ccc}%
\operatorname{Re}\left[  m\left(  m-1\right)  Z^{m-2}\right]  x_{3} &
-\operatorname{Im}\left[  m\left(  m-1\right)  Z^{m-2}\right]  x_{3} &
\operatorname{Re}mZ^{m-1}\\
& -\operatorname{Re}\left[  m\left(  m-1\right)  Z^{m-2}\right]  x_{3} &
-\operatorname{Im}mZ^{m-1}\\
&  & 0
\end{array}
\right]  .
\]
It follows that%
\[
\sigma_{2}\left(  D^{2}h\right)  =-\left[  m\left(  m-1\right)  \rho
^{m-2}\right]  ^{2}x_{3}^{2}-m^{2}\rho^{2m-2}.
\]
Thus%
\[
\tilde{\bigtriangleup}Q+\sigma_{2}\left(  D^{2}h\right)  =\left[  m\left(
m-1\right)  \rho^{m-2}\right]  ^{2}x_{3}^{2}+m^{2}\rho^{2m-2}+\sigma
_{2}\left(  D^{2}h\right)  =0.
\]
Finally we fix the \textquotedblleft harmonic\textquotedblright\ $H$
satisfying $\tilde{\bigtriangleup}H=0$ with%
\begin{align*}
a_{0}  &  =-1\\
a_{j}  &  =-\frac{2\cdot\left(  2m-2j+2\right)  \left(  2m-2j+1\right)
}{\left(  2j\right)  ^{2}}a_{j-1}\\
&  =\left(  -1\right)  ^{j+1}\frac{2^{j}2m\left(  2m-1\right)  \cdots\left(
2m-2j+1\right)  }{2^{2}4^{2}\cdots\left(  2j\right)  ^{2}}\ \ \text{for }%
j\geq1,
\end{align*}
and $\varepsilon$ is still pending. Therefore, $P=\frac{1}{2}\left(  x_{1}%
^{2}+x_{2}^{2}\right)  +h+Q+H,$ satisfies
\[
\sigma_{2}\left(  D^{2}P\right)  -1=\left[  r^{3m-3}\right]  .
\]

Now the analytic solution $u$ to (\ref{Eck}) with initial data $P$ follows
from Cauchy-Kowalevskaya. As in [NV], considering the linear equation for
difference $u-P,$ the Cauchy-Kowalevskaya procedure implies that the solution
$u$ coincides with $P$ up to order $3m-2$ ($\geq2m$ for $m\geq2$). Thus
Property 2.1 is verified.

We move to Property 2.2. We have
\begin{equation}
D^{2}u=\left[
\begin{array}
[c]{ccc}%
1+\left[  r^{m-1}\right]  & \left[  r^{m-1}\right]  & \operatorname{Re}%
mZ^{m-1}+\left[  r^{2m-2}\right] \\
& 1+\left[  r^{m-1}\right]  & -\operatorname{Im}mZ^{m-1}+\left[
r^{2m-2}\right] \\
&  & \frac{m^{2}}{2}\rho^{2m-2}+H_{33}+\left[  r^{2m-1}\right]
\end{array}
\right]  . \label{AsympHess}%
\end{equation}
Because the eigenvalues are Lipschitz functions of the matrix entries, we get%
\begin{align*}
\lambda_{1}  &  =1+\left[  r^{m-1}\right] \\
\lambda_{2}  &  =1+\left[  r^{m-1}\right]  .
\end{align*}
By the quadratic Taylor expansion of the isolated eigenvalue $\lambda_{3}$ in
terms of the matrix entries near $D^{2}u\left(  0\right)  ,$ we obtain%

\begin{align*}
\lambda_{3}  &  =u_{33}-u_{13}^{2}-u_{23}^{2}+\left[  r^{3m-3}\right] \\
&  =\left.
\begin{array}
[c]{c}%
\frac{m^{2}}{2}\rho^{2m-2}+\varepsilon\sum_{j=0}^{m}\left(  2m-2j\right)
\left(  2m-2j-1\right)  a_{j}x_{3}^{2m-2j}\rho^{2j}\\
-m^{2}\rho^{2m-2}+\left[  r^{2m-1}\right]  \ \ \text{for }m\geq2
\end{array}
\right. \\
&  =\left.
\begin{array}
[c]{c}%
\varepsilon\left[  -2m\left(  2m-1\right)  x_{3}^{2m-2}+\tilde{a}_{2}%
x_{3}^{2m-4}\rho^{2}+\cdots+\tilde{a}_{m-1}\rho^{2m-2}\right]  -\frac{m^{2}%
}{2}\rho^{2m-2}\\
+\left[  r^{2m-1}\right]
\end{array}
\right. \\
&  =H_{33}-\frac{m^{2}}{2}\rho^{2m-2}+\left[  r^{2m-1}\right]  .
\end{align*}
The \textquotedblleft harmonic\textquotedblright\ function $H_{33}$ cannot
have a definite sign near the origin, but with the help of $-\frac{m^{2}}%
{2}\rho^{2m-2}$ and small $\varepsilon,$ we make $\lambda_{3}$ negative. Let
$\eta$ be a small positive constant to be chosen shortly.

\noindent Case 1: $\eta\left\vert x_{3}\right\vert \geq\rho.$ We have%
\[
\left[  -2m\left(  2m-1\right)  x_{3}^{2m-2}+\tilde{a}_{2}x_{3}^{2m-4}\rho
^{2}+\cdots\right]  =-\left[  2m\left(  2m-1\right)  +O\left(  1\right)
\eta^{2}\right]  x_{3}^{2m-2}.
\]
Note $r/\sqrt{1+\eta^{2}}\leq\left\vert x_{3}\right\vert \leq r,$ then
\[
-\left\{
\begin{array}
[c]{c}%
\varepsilon\left[  2m\left(  2m-1\right)  +O\left(  1\right)  \eta^{2}\right]
\\
+\frac{m^{2}}{2}+o\left(  1\right)
\end{array}
\right\}  r^{2m-2}\leq\lambda_{3}\leq-\varepsilon\left[
\begin{array}
[c]{c}%
\frac{2m\left(  2m-1\right)  +O\left(  1\right)  \eta^{2}}{\left(
\sqrt{1+\eta^{2}}\right)  ^{2m-2}}\\
+o\left(  1\right)
\end{array}
\right]  r^{2m-2}.
\]

\noindent Case 2: $\eta\left\vert x_{3}\right\vert <\rho.$ Note $r\eta
/\sqrt{1+\eta^{2}}\leq\rho\leq r,$ we have%
\[
\left[  -2m\left(  2m-1\right)  x_{3}^{2m-2}+\tilde{a}_{2}x_{3}^{2m-4}\rho
^{2}+\cdots\right]  =\frac{O\left(  1\right)  }{\eta^{2m-2}}\rho^{2m-2},
\]
then%
\[
\lambda_{3}=-\left[  \frac{m^{2}}{2}-\frac{\varepsilon O\left(  1\right)
}{\eta^{2m-2}}\right]  \rho^{2m-2}+\left[  r^{2m-1}\right]
\]
and%
\[
-\left[
\begin{array}
[c]{c}%
\frac{m^{2}}{2}-\frac{\varepsilon O\left(  1\right)  }{\eta^{2m-2}}\\
+o\left(  1\right)
\end{array}
\right]  r^{2m-2}\leq\lambda_{3}\leq\left\{
\begin{array}
[c]{c}%
-\left[  \frac{m^{2}}{2}-\frac{\varepsilon O\left(  1\right)  }{\eta^{2m-2}%
}\right]  \frac{\eta^{2m-2}}{\left(  \sqrt{1+\eta^{2}}\right)  ^{2m-2}}\\
+o\left(  1\right)
\end{array}
\right\}  r^{2m-2}.
\]

We first choose $\eta=\eta\left(  m\right)  >0$ small, next $\varepsilon
=\varepsilon\left(  \eta,m\right)  >0$ smaller, then there exist $\delta
_{1}=\delta\left(  \eta,m\right)  >0$ and $\delta_{2}=\delta_{2}\left(
m\right)  >0$ such that%
\[
-\delta_{2}r^{2m-2}\leq\lambda_{3}\leq-\delta_{1}r^{2m-2}%
\]
for $r\leq r_{m}.$ Here $r_{m}$ is within the valid radius for the
Cauchy-Kowalevskaya solution $u.$

Property 2.3 follows from $u=P+\left[  r^{3m-2}\right]  .$

Finally we prove Property 2.4. The upper bound is straightforward. For the
lower bound, from Property 2.3, we have%

\begin{align*}
|Du(x)|^{2}  &  =(x_{1}+[r^{m}])^{2}+(x_{2}+[r^{m}])^{2})^{2}\\
&  +(ReZ^{m}+\frac{m^{2}}{2}\rho^{2m-2}x_{3}-2m\varepsilon x_{3}%
^{2m-1}+\varepsilon\rho^{2}[r^{2m-2}]+[r^{2m}])^{2}.
\end{align*}

Case 1: $x_{3}^{2}\geq\rho.$ From $r^{2}=\rho^{2}+x_{3}^{2}\leq\left(
x_{3}^{2}+1\right)  x_{3}^{2},$ we know%
\[
\left\vert x_{3}\right\vert \geq r.
\]
Note that the other terms than $-2m\varepsilon x_{3}^{2m-1}$ in $u_{3}\left(
x\right)  $ have the following asymptotic behavior near the origin
\begin{align*}
\left\vert ReZ^{m}\right\vert  &  \leq\rho^{m}=x_{3}^{2m},\\
\left\vert \frac{m^{2}}{2}\rho^{2m-2}x_{3}\right\vert  &  \leq\frac{m^{2}}%
{2}\left\vert x_{3}\right\vert ^{4m-3},\\
\varepsilon\rho^{2}[r^{2m-2}]  &  =O(x_{3}^{2m+2}),\\
\lbrack r^{2m}]  &  =O(x_{3}^{2m}).
\end{align*}
It follows that%
\begin{align*}
\left\vert Du\left(  x\right)  \right\vert ^{2}  &  \geq\left\vert
u_{3}\left(  x\right)  \right\vert ^{2}=\left[  -2m\varepsilon x_{3}%
^{2m-1}+O(x_{3}^{2m})\right]  ^{2}\\
&  \geq\delta_{3}\left(  m\right)  x_{3}^{2\left(  2m-1\right)  }\geq
\delta_{3}\left(  m\right)  r^{2\left(  2m-1\right)  }%
\end{align*}
for $\left\vert x\right\vert \leq r_{m}$ with positive $r_{m}$ and $\delta
_{3}\left(  m\right)  $ to be fixed shortly.

Case 2: $x_{3}^{2}<\rho.$ From $r^{2}=\rho^{2}+x_{3}^{2}\leq\left(
\rho+1\right)  \rho,$ we know
\[
\rho>r^{2}.
\]
Then
\begin{align*}
\left\vert Du\left(  x\right)  \right\vert ^{2}  &  \geq u_{1}^{2}\left(
x\right)  +u_{2}^{2}\left(  x\right)  =\rho^{2}+2x_{1}[r^{m}]+2x_{2}%
[r^{m}]+2[r^{m}]^{2}\\
&  =\rho^{2}+O(\rho^{\frac{m+1}{2}})\\
&  \geq\frac{1}{2}\rho^{2}\geq\frac{1}{2}r^{4}\geq r^{2\left(  2m-1\right)  }%
\end{align*}
for $\rho\leq r\leq r_{m}$ with the positive $r_{m}$ to be fixed next.

Now we choose positive $\delta_{3}\left(  m\right)  $ small and the small
positive $r_{m}$ within the valid radius for Cauchy-Kowalevskaya solution $u$
and Property 2.2, Property 2.4 is then completely justified.

Since $u\left(  r_{m}x\right)  /r_{m}^{2}$ is still a solution to $\sigma
_{2}\left(  D^{2}u\right)  =1$ in $B_{1}\subset R^{3}.$ We may assume the
above constructed solution is already defined in $B_{1}\subset\mathbb{R}^{3}.$
Note that $D\left[  u\left(  r_{m}x\right)  /r_{m}^{2}\right]  =Du\left(
r_{m}x\right)  /r_{m}$ and $D^{2}\left[  u\left(  r_{m}x\right)  /r_{m}%
^{2}\right]  =D^{2}u\left(  r_{m}x\right)  ,$ we see that Property 2.2 and
Property 2.4 are still valid in $B_{1}$ with $\delta_{1},\ \delta_{2}%
,\ \delta_{3}$ replaced by $r_{m}^{2m-2}\delta_{1},$ $r_{m}^{2m-2}\delta
_{2},\ $ $r_{m}^{2m-2}\delta_{2}$ respectively, and $\delta_{4}$ unchanged.

\section{Rotate to subcritical phases $\left\vert \Theta\right\vert <\frac
{\pi}{2}:$ proof of Theorem 1.1}

In this section, we carry out the construction of the singular solutions in
Theorem 1.1 by \textquotedblleft horizontally\textquotedblright\ and $\pi/2$
rotating the Cauchy-Kowalevskaya solutions from Section 2. The latter
rotation, Proposition 3.1 is pivotal.

Step 1. Let $\alpha\in\lbrack0,\pi/4).$ We will take $\alpha=\Theta/2$ for
$\Theta\in\lbrack0,\pi/2)$ in Step 3 of this section. We make a $U\left(
3\right)  $ rotation in $\mathbb{C}^{3}:$ $\tilde{z}^{\prime}=e^{\alpha
\sqrt{-1}}z^{\prime}$ and $\tilde{z}_{3}=z_{3}$ with $\tilde{z}=\left(
\tilde{z}^{\prime},\tilde{z}_{3}\right)  =\left(  \tilde{x}^{\prime},\tilde
{x}_{3}\right)  +\sqrt{-1}\left(  \tilde{y}^{\prime},\tilde{y}_{3}\right)  $
and $z=\left(  z^{\prime},z_{3}\right)  =$ $\left(  x^{\prime},x_{3}\right)
+\sqrt{-1}\left(  y^{\prime},y_{3}\right)  .$ Because $U\left(  3\right)  $
rotations preserve the length and complex structure, $\mathfrak{M}=\left(
x,Dv\left(  x\right)  \right)  $ for $x\in B_{1}$ is still a special
Lagrangian submanifold in the new coordinate system with parameterization%
\begin{equation}
\left\{
\begin{array}
[c]{c}%
\tilde{x}=\left(  x_{1}\cos\alpha+u_{1}\left(  x\right)  \sin\alpha
,\ x_{2}\cos\alpha+u_{2}\left(  x\right)  \sin\alpha,\ x_{3}\right) \\
\tilde{y}=\left(  -x_{1}\sin\alpha+u_{1}\left(  x\right)  \cos\alpha
,\ -x_{2}\sin\alpha+u_{2}\left(  x\right)  \cos\alpha,\ u_{3}\left(  x\right)
\right)
\end{array}
\right.  . \label{h-rotation}%
\end{equation}
We show that $\mathfrak{M}$ is also a \textquotedblleft
gradient\textquotedblright\ graph over $\tilde{x}$ space. From Property 2.2,
we know that $u\left(  x^{\prime},x_{3}\right)  $ is a convex function in
terms of $x^{\prime}$ for $\left\vert x\right\vert \leq1,$ or if necessary
$\left\vert x\right\vert \leq r_{m}$ with $r_{m}$ depending only on $m.$ From
(\ref{AsympHess}) we also assume $\left\vert D^{\prime}u_{3}\left(  x\right)
\right\vert =\left\vert \left(  u_{13},u_{23}\right)  \left(  x\right)
\right\vert \leq1/2$ for $\left\vert x\right\vert \leq r_{m}.$ Then we have%
\begin{gather}
\delta_{5}\left(  m\right)  \left\vert x-x^{\ast}\right\vert ^{2}%
\geq\left\vert \tilde{x}\left(  x\right)  -\tilde{x}\left(  x^{\ast}\right)
\right\vert ^{2}\label{distanceshrinking}\\
=\left\vert \left(  x^{\prime}-x^{\ast\prime}\right)  \cos\alpha+\left[
\begin{array}
[c]{c}%
D^{\prime}u\left(  x^{\prime},x_{3}\right)  -D^{\prime}u\left(  x^{\prime
},x_{3}^{\ast}\right) \\
+\underline{D^{\prime}u\left(  x^{\prime},x_{3}^{\ast}\right)  -D^{\prime
}u\left(  x^{\ast\prime},x_{3}^{\ast}\right)  }%
\end{array}
\right]  \sin\alpha\right\vert ^{2}+\left\vert x_{3}-x_{3}^{\ast}\right\vert
^{2}\nonumber\\
\geq\left[
\begin{array}
[c]{c}%
\frac{1}{2}\left\vert \left(  x^{\prime}-x^{\ast\prime}\right)  \cos
\alpha+\left(  \underline{u\left(  x^{\prime},x_{3}^{\ast}\right)  -D^{\prime
}u\left(  x^{\ast\prime},x_{3}^{\ast}\right)  }\right)  \sin\alpha\right\vert
^{2}\\
-\left\vert \left(  D^{\prime}u\left(  x^{\prime},x_{3}\right)  -D^{\prime
}u\left(  x^{\prime},x_{3}^{\ast}\right)  \right)  \sin\alpha\right\vert
^{2}+\left\vert x_{3}-x_{3}^{\ast}\right\vert ^{2}%
\end{array}
\right] \nonumber\\
\geq\left[
\begin{array}
[c]{c}%
\frac{\cos^{2}\alpha}{2}\left\vert x^{\prime}-x^{\ast\prime}\right\vert
^{2}+\cos\alpha\sin\alpha\ \underset{\geq0}{\underbrace{\left\langle
x^{\prime}-x^{\ast\prime},D^{\prime}u\left(  x^{\prime},x_{3}^{\ast}\right)
-D^{\prime}u\left(  x^{\ast\prime},x_{3}^{\ast}\right)  \right\rangle }}\\
-\sin^{2}\alpha\ \underset{\leq1}{\underbrace{2\left\Vert D^{\prime}%
u_{3}\right\Vert _{L^{\infty}\left(  B_{r_{m}}\right)  }}}\ \left\vert
x_{3}-x_{3}^{\ast}\right\vert ^{2}+\left\vert x_{3}-x_{3}^{\ast}\right\vert
^{2}%
\end{array}
\right] \nonumber\\
\geq\frac{\cos^{2}\alpha}{2}\left\vert x^{\prime}-x^{\ast\prime}\right\vert
^{2}+\left(  1-\sin^{2}\alpha\right)  \left\vert x_{3}-x_{3}^{\ast}\right\vert
^{2}\nonumber\\
\geq\frac{1}{4}\left\vert x-x^{\ast}\right\vert ^{2}.
\label{distanceexpansion}%
\end{gather}
It follows that $\mathfrak{M}$ is a special Lagrangian graph $\left(
\tilde{x},D\tilde{u}\left(  \tilde{x}\right)  \right)  $ over a domain
containing a ball of radius $1/\sqrt{2}$ in $\tilde{x}$ space. The Hessian of
the potential function $\tilde{u}$ satisfies%
\begin{gather}
D^{2}\tilde{u}=\frac{\partial\tilde{y}}{\partial\tilde{x}}=\frac
{\partial\tilde{y}}{\partial x}\left(  \frac{\partial\tilde{x}}{\partial
x}\right)  ^{-1}\nonumber\\
=\left[
\begin{array}
[c]{ccc}%
-\sin\alpha+u_{11}\cos\alpha & u_{12}\cos\alpha & u_{13}\cos\alpha\\
u_{12}\cos\alpha & -\sin\alpha+u_{22}\cos\alpha & u_{23}\cos\alpha\\
u_{13} & u_{23} & u_{33}%
\end{array}
\right] \nonumber\\
\ \ \ \ \ \ \ \ \ \ \left[
\begin{array}
[c]{ccc}%
\cos\alpha+u_{11}\sin\alpha & u_{12}\sin\alpha & u_{13}\sin\alpha\\
u_{12}\sin\alpha & \cos\alpha+u_{22}\sin\alpha & u_{23}\sin\alpha\\
0 & 0 & 1
\end{array}
\right]  ^{-1}\nonumber\\
=\left[
\begin{array}
[c]{ccc}%
\tan\left(  \frac{\pi}{4}-\alpha\right)  &  & \\
& \tan\left(  \frac{\pi}{4}-\alpha\right)  & \\
&  & 0
\end{array}
\right]  +\left[  r^{m-1}\right]  \label{h-expansion}%
\end{gather}
and%
\begin{equation}
\det D^{2}\tilde{u}=\tan\left(  \frac{\pi}{4}-\alpha\right)  \left[
-\frac{m^{2}}{2}\rho^{2m-2}+\tan\left(  \frac{\pi}{4}-\alpha\right)
H_{33}\right]  -\left[  r^{2m-1}\right]  , \label{det-expression}%
\end{equation}
where the above abused notation $\left[  r^{m-1}\right]  $ also represents a
matrix whose all entries are analytic functions starting from order $m-1,$ and
(\ref{h-expansion}) (\ref{det-expression}) follow from a simple calculation
and the asymptotic behavior of $D^{2}u,$ (\ref{AsympHess}). We verify the
following three properties for $D^{2}\tilde{u}.$ There exists a positive
number $\tilde{r}_{m,\alpha}$ depending only on $m$ and $\alpha\in\lbrack
0,\pi/4)$ such that for $\left\vert \tilde{x}\right\vert \leq$ $\tilde
{r}_{m,\alpha}$ we have:

\noindent Property 3.1. The determinant $\det D^{2}\tilde{u}\left(  \tilde
{x}\right)  $ is negative for small $\tilde{x}\neq0,$ indeed%
\[
\det D^{2}\tilde{u}\left(  \tilde{x}\right)  \approx-\tan\left(  \frac{\pi}%
{4}-\alpha\right)  \left\vert \tilde{x}\right\vert ^{2m-2};
\]

\noindent Property 3.2. The upper left 2$\times$2 principle minor of the
Hessian $D^{2}\tilde{u},\ $%
\[
2\tan\left(  \frac{\pi}{4}-\alpha\right)  I\geq\left(  D^{2}\tilde{u}\right)
^{\prime}\geq\frac{\tan\left(  \frac{\pi}{4}-\alpha\right)  }{2}I;
\]

\noindent Property 3.3. The three eigenvalues $\tilde{\lambda}_{i}$ of the
Hessian $D^{2}\tilde{u}$ satisfy
\[
\left\{
\begin{array}
[c]{l}%
\tilde{\theta}_{1}=\arctan\tilde{\lambda}_{1}=\left(  \frac{\pi}{4}%
-\alpha\right)  \left[  1+O\left(  \left\vert \tilde{x}\right\vert
^{m-1}\right)  \right] \\
\tilde{\theta}_{2}=\arctan\tilde{\lambda}_{2}=\left(  \frac{\pi}{4}%
-\alpha\right)  \left[  1+O\left(  \left\vert \tilde{x}\right\vert
^{m-1}\right)  \right] \\
\tilde{\theta}_{3}=\arctan\tilde{\lambda}_{3}\approx-\frac{1}{\tan\left(
\frac{\pi}{4}-\alpha\right)  }\left\vert \tilde{x}\right\vert ^{2m-2}\left[
1+O\left(  \left\vert \tilde{x}\right\vert ^{m-1}\right)  \right]
\end{array}
\right.  ;
\]
where $``\approx^{\prime\prime}$ means two quantities are comparable up to a
multiple of constant depending only on $m$ and $\alpha.$ Relying on
(\ref{det-expression}), repeating the arguments for the estimate of
$\lambda_{3}$ in Section 2, using (\ref{distanceexpansion}) and
(\ref{h-rotation}), we obtain Property 3.1. Property 3.2 follows from
(\ref{h-expansion}). From (\ref{distanceexpansion}) (\ref{h-expansion}) and
the Lipschitz continuity of eigenvalues in terms of matrix entries, we derive
the estimates for the first two eigenvalues in Property 3.3. In turn, noticing
$\tilde{\lambda}_{3}=\det D^{2}\tilde{u}/\left(  \tilde{\lambda}_{1}%
\tilde{\lambda}_{2}\right)  ,$ relying on both (\ref{distanceshrinking}) and
(\ref{distanceexpansion}) we get two sided estimates of the last eigenvalue.

Step 2. We proceed with the following proposition.

\begin{proposition}
Let $\mathcal{L}=\left(  x,Df\right)  $ be a Lagrangian surface in
$\mathbb{C}^{3}=\mathbb{R}^{3}\times\mathbb{R}^{3}$ with the smooth potential
$f$ over $B_{\rho}\subset\mathbb{R}^{3},$ satisfying:%
\begin{gather}
Df\left(  0\right)  =0,\nonumber\\
\det D^{2}f\left(  x\right)  <0\ \ \text{for }x\neq0,\nonumber\\
\left\{
\begin{array}
[c]{c}%
\kappa^{-1}I\geq\left[
\begin{array}
[c]{cc}%
f_{11}\left(  x\right)  & f_{12}\left(  x\right) \\
f_{21}\left(  x\right)  & f_{22}\left(  x\right)
\end{array}
\right]  \geq\kappa I\\
\left\vert D^{\prime}f_{3}\left(  x\right)  \right\vert =\left\vert \left(
f_{13},f_{23}\right)  \left(  x\right)  \right\vert \leq\frac{1}%
{2},\ \text{say}%
\end{array}
\right\}  \ \text{\ for }x\in B_{\rho}. \label{hconvexity}%
\end{gather}
Then $\mathcal{L}$ can be re-represented as a graph $\left(  \tilde{x}%
,\tilde{y}\right)  =\left(  \tilde{x},D\tilde{f}\left(  \tilde{x}\right)
\right)  $ over open set $\Omega=Df\left(  B_{\frac{1}{2}\kappa^{2}\rho
}\right)  $ with $\tilde{x}+\sqrt{-1}\tilde{y}=e^{-\frac{\pi}{2}\sqrt{-1}%
}\left(  x+\sqrt{-1}y\right)  $ and $\tilde{f}\in C^{1}\left(  \Omega\right)
\cap C^{\infty}\left(  \Omega\backslash\left\{  0\right\}  \right)  .$
\end{proposition}

\begin{proof}
[Proof of Proposition 3.1]Note that the $U\left(  3\right)  $ rotation by
$\pi/2$ is $\left(  \tilde{x},\tilde{y}\right)  $ $=\left(  y,-x\right)  .$
This proposition really says that the map $Df$ \ has a (unique) continuous
inverse $\Phi=-D\tilde{f}.$

Step 2.1. We first prove $Df$ is one-to-one on $B_{\kappa^{2}\rho}.$ Consider
a coordinate change given by $t=\Psi\left(  x\right)  =(f_{1}\left(  x\right)
,f_{2}\left(  x\right)  ,x_{3}).$ Then the Jacobian of $\Psi$ is
\begin{equation}
\det D_{x}\Psi\left(  x\right)  =\det\left[
\begin{array}
[c]{ccc}%
f_{11} & f_{12} & f_{13}\\
f_{21} & f_{22} & f_{23}\\
0 & 0 & 1
\end{array}
\right]  \left(  x\right)  =\det\left[
\begin{array}
[c]{cc}%
f_{11} & f_{12}\\
f_{21} & f_{22}%
\end{array}
\right]  \left(  x\right)  >0. \label{hnonvanishing}%
\end{equation}
Hence $\Psi$ is a local diffeomorphism on $B_{\rho}.$ Note that $\Psi$ is
actually a distance expansion map. We have for all $x,\ x^{\#}$ in $B_{\rho}$
\begin{gather}
\left\vert \Psi\left(  x\right)  -\Psi\left(  x^{\#}\right)  \right\vert
^{2}=\left\vert
\begin{array}
[c]{c}%
D^{\prime}f\left(  x^{\prime},x_{3}\right)  -D^{\prime}f\left(  x^{\#\prime
},x_{3}\right) \\
+D^{\prime}f\left(  x^{\#\prime},x_{3}\right)  -D^{\prime}f\left(
x^{\#\prime},x_{3}^{\#}\right)
\end{array}
\right\vert ^{2}+\left\vert x_{3}-x_{3}^{\#}\right\vert ^{2}\nonumber\\
\geq\left[
\begin{array}
[c]{c}%
\frac{1}{2}\left\vert D^{\prime}f\left(  x^{\prime},x_{3}\right)  -D^{\prime
}f\left(  x^{\#\prime},x_{3}\right)  \right\vert ^{2}-\left\vert D^{\prime
}f\left(  x^{\#\prime},x_{3}\right)  -D^{\prime}f\left(  x^{\#\prime}%
,x_{3}^{\#}\right)  \right\vert ^{2}\\
+\left\vert x_{3}-x_{3}^{\#}\right\vert ^{2}%
\end{array}
\right] \nonumber\\
=\left[
\begin{array}
[c]{c}%
\frac{1}{2}\left\vert D^{\prime}f\left(  x^{\prime},x_{3}\right)  -D^{\prime
}f\left(  x^{\#\prime},x_{3}\right)  -\kappa\left(  x^{\prime}-x^{\#\prime
}\right)  +\kappa\left(  x^{\prime}-x^{\#\prime}\right)  \right\vert ^{2}\\
-\left\vert D^{\prime}f\left(  x^{\#\prime},x_{3}\right)  -D^{\prime}f\left(
x^{\#\prime},x_{3}^{\#}\right)  \right\vert ^{2}+\left\vert x_{3}-x_{3}%
^{\#}\right\vert ^{2}%
\end{array}
\right] \nonumber\\
\geq\,\,\left[
\begin{array}
[c]{c}%
\underset{\geq0}{\,\underbrace{\left\langle D^{\prime}f\left(  x^{\prime
},x_{3}\right)  -D^{\prime}f\left(  x^{\#\prime},x_{3}\right)  -\kappa\left(
x^{\prime}-x^{\#\prime}\right)  ,\kappa\left(  x^{\prime}-x^{\#\prime}\right)
\right\rangle }}\\
+\frac{1}{2}\left\vert \kappa\left(  x^{\prime}-x^{\#\prime}\right)
\right\vert ^{2}-2\left\Vert D^{\prime}f\right\Vert _{L^{\infty}\left(
B_{\rho}\right)  }^{2}\left\vert x_{3}-x_{3}^{\#}\right\vert ^{2}+\left\vert
x_{3}-x_{3}^{\#}\right\vert ^{2}%
\end{array}
\right] \nonumber\\
\geq\frac{\kappa^{2}}{2}\left\vert x^{\prime}-x^{\#\prime}\right\vert
^{2}+\left(  1-2\left\Vert D^{\prime}f\right\Vert _{L^{\infty}\left(  B_{\rho
}\right)  }^{2}\right)  \left\vert x_{3}-x_{3}^{\#}\right\vert ^{2}\nonumber\\
\geq\frac{\kappa^{2}}{2}\left\vert x-x^{\#}\right\vert ^{2},
\label{Kexpansion}%
\end{gather}
where we used (\ref{hconvexity}). Thus $\Psi$ is a \textquotedblleft
global\textquotedblright\ diffeomorphism on $B_{\rho}.$

We claim that the $\Psi$-image of $B_{\rho},$ $\Psi(B_{\rho})\supset
B_{\frac{\kappa}{\sqrt{2}}\rho}^{t}.$ Otherwise, let $t^{\#}$ be a boundary
point of $\Psi(B_{\rho})$ in $\mathring{B}_{\frac{\kappa}{\sqrt{2}}\rho}^{t}.$
We know there exist a sequence of points $x_{i}\in B_{\rho}$ such that
$\Psi(x_{i})\ $goes to $t^{\#}$ and $x_{i}$ goes to $x^{\#}\in\bar{B}_{\rho}$
as $i$ goes to infinity. If $x^{\#}\in\mathring{B}_{\rho},$ then $\Psi\left(
x^{\#}\right)  =t^{\#}$ by the continuity of $\Psi.$ But this is impossible
because $\Psi\left(  x^{\#}\right)  $ is an interior point of $\Psi(B_{\rho})$
under the diffeomorphism of $\Psi.$ If $x^{\#}\in\partial B_{\rho},$ from
(\ref{Kexpansion}), we have%
\[
|t^{\#}|=\lim_{i\rightarrow\infty}\left\vert \Psi(x_{i})-0\right\vert
\ \geq\lim_{i\rightarrow\infty}\frac{\kappa}{\sqrt{2}}|x_{i}-0|=\frac{\kappa
}{\sqrt{2}}|x^{\#}|=\frac{\kappa}{\sqrt{2}}\rho.
\]
This contradicts $t^{\#}\in\mathring{B}_{\frac{\kappa}{\sqrt{2}}\rho}^{t}.$

From (\ref{hconvexity}) (only the upper bound), then%
\[
|\Psi(x)-\Psi(x^{\#})|\leq||DD^{^{\prime}}f||_{L^{\infty}\left(  B_{\rho
}\right)  }\ |x-x^{\#}|\leq\kappa^{-1}|x-x^{\#}|,
\]
if follows that $\Psi^{-1}$ is also a distance expansion map with a factor
$\kappa.$ Apply the arguments above we get $\Psi^{-1}(B_{\frac{1}{2}\kappa
\rho}^{t})\supset B_{\frac{1}{2}\kappa^{2}\rho}$ or $B_{\frac{1}{2}\kappa\rho
}^{t}\supset\Psi(B_{\frac{1}{2}\kappa^{2}\rho}).$

Now for the injectivity of $Df$ on $B_{\frac{1}{2}\kappa^{2}\rho},$ it
suffices to show that
\[
y\left(  t\right)  =Df\circ\Psi^{-1}\left(  t\right)  =(t_{1},t_{2}%
,\ f_{3}(x\left(  t\right)  )
\]
is one-to-one in $B_{\frac{1}{2}\kappa\rho}^{t}.$ Suppose that $y\left(
t\right)  =y\left(  t^{\#}\right)  ,$ then
\[
t_{1}^{\#}=t_{1},\ t_{2}^{\#}=t_{2},\ \ y_{3}(t^{\#})=y_{3}(t).
\]
Note that
\begin{align*}
\frac{\partial y_{3}\left(  t_{1},t_{2},\xi\right)  }{\partial t_{3}}  &
=\left[
\begin{array}
[c]{ccc}%
1 & 0 & 0\\
0 & 1 & 0\\
\ast & \ast & \frac{\partial y_{3}\left(  t_{1},t_{2},\xi\right)  }{\partial
t_{3}}%
\end{array}
\right] \\
&  =\det D_{t}\left(  Df\circ\Psi^{-1}\right)  =\left.  \det\left(
D^{2}f\right)  \right\vert _{\Psi^{-1}\left(  t_{1},t_{2},\xi\right)  }%
\cdot\left.  \det D_{t}\Psi^{-1}\right\vert _{\left(  t_{1},t_{2},\xi\right)
}<0
\end{align*}
for $(t_{1},t_{2},\xi)\neq0,$ where we used (\ref{hnonvanishing}) and $\det
D^{2}f\left(  x\right)  <0$ for $x\neq0.$ It follows that the function
$y_{3}(t_{1},t_{2},\xi)$ is strictly decreasing in $\xi.$ Now $y_{3}%
(t_{1},t_{2},t_{3}^{\#})=y_{3}(t_{1},t_{2},t_{3})$ implies $t_{3}^{\#}=t_{3}.$
This shows that $y=Df\circ\Psi^{-1}$ is one-to-one.

So far we have obtained the inverse function $\Phi=\left(  Df\right)  ^{-1}$
on $\Omega=Df\left(  B_{\frac{1}{2}\kappa^{2}\rho}\right)  .$

Step 2.2. We prove $Df$ is an open map from $B_{\rho}$ to $\mathbb{R}^{3}.$
Since the Jacobian $\det D^{2}f\left(  x\right)  \neq0$ for $x\neq0,$ $Df$ is
already a local diffeomorphism for $x\neq0.$ It suffices to show that the
image of an open neighborhood of $0$ in $B_{\rho},$ under $Df,$ contains an
open neighborhood of $0$ in $\mathbb{R}^{3}.$ Since $\Psi$ is a
diffeomorphism, we only need to show this property for $Df\circ\Psi^{-1}.$
Indeed we only need to consider the image of the ball $B_{2\eta}^{t}$ of
radius $2\eta$ centered at $t=0$ for a small $\eta>0.$ According to Step 2.1,
$y\left(  t_{1},t_{2},\cdot\right)  $ is strictly decreasing in the third
variable. So $2h_{-}=y_{3}(0,0,\eta)<0$ and $2h_{+}=y_{3}(0,0,-\eta)>0.$ By
continuity of $y=Df\circ\Psi^{-1},$ there exists $\eta^{\prime}\in\left(
0,\eta\right)  $ such that $y_{3}(t_{1},t_{2},\eta)<h_{-}<0\ \ $and
$y_{3}(t_{1},t_{2},-\eta)>h_{+}>0$ for $|(t_{1},t_{2})|\leq\eta^{\prime}.$
Then by intermediate value theorem (for function $y_{3}\left(  t_{1}%
,t_{2},\cdot\right)  $), the open set%
\begin{align*}
\{|(y_{1},y_{2})|  &  <\eta^{\prime}\}\times\{h_{-}<y_{3}<h_{+}\}\subset
Df\circ\Psi^{-1}(\{|(t_{1},t_{2})|\leq\eta^{\prime}\}\times\{|t_{3}|\leq
\eta\})\\
&  \subset Df\circ\Psi^{-1}(B_{2\eta}^{t}).
\end{align*}
Thus $Df$ is an open map.

Step 2.3. Now $\Omega=Df(B_{\kappa^{2}\rho})$ is an open neighborhood of
$y=0,$ and $\Phi$ is continuous on $\Omega$ by the openness of $Df.$ Lastly we
find a potential for the Lagrangian submanifold $\mathcal{L}$ now represented
as $\left(  \tilde{x},-\Phi\left(  \tilde{x}\right)  \right)  .$ Let%
\[
\tilde{f}(\tilde{x})=\int_{0}^{\tilde{x}}-\Phi^{1}(s)ds_{1}-\Phi^{2}%
(s)ds_{2}-\Phi^{3}(s)ds_{3}.
\]
Because $D_{\tilde{x}}\left(  -\Phi\left(  \tilde{x}\right)  \right)  $ is
symmetric $-\left(  D^{2}f\right)  ^{-1}$ when $\tilde{x}\neq0$ and $\Phi$ is
bounded, $\tilde{f}\left(  \tilde{x}\right)  $ is well-defined on $\Omega.$
Further we know $\tilde{f}\in C^{1}\left(  \Omega\right)  \cap C^{\infty
}\left(  \Omega\backslash\left\{  0\right\}  \right)  .$

The proof for Proposition 3.1 is complete.
\end{proof}

\textbf{Remark.} For the purpose of Theorem 1.1, we can replace
(\ref{hconvexity}) by a weaker condition (\ref{hnonvanishing}) $\det\left[
D^{2}f\right]  ^{\prime}>0.$ Consequently we have no estimate on the size of
the existing neighborhood supporting the solution $\tilde{\tilde{u}}$ in this
section, then $u^{m}$ for each single $m$ and $\Theta.$ The stronger
assumption (\ref{hconvexity}) is designed for Theorem 1.2 where we need a
uniform control with respect to $\varepsilon$ on the valid radius for the
solutions $\tilde{\tilde{u}}^{\varepsilon},$ then $u^{\varepsilon}.$ Lastly
there is another argument for the openness of the particular map $D\tilde
{u}=Df,$ relying on the uniqueness of the pre-image of $y=0$ (which can be
also derived from the distance expansion property at the origin
(\ref{mismatchpower})), instead of using the strict monotonicity property in
Step 2.2.

Step 3. Equipped with Property 3.1, 3.2, and (\ref{h-expansion}), we apply
Proposition 3.1 to our function $\tilde{u}\left(  \tilde{x}\right)  $ with
$\tilde{x}$ replaced by $\tilde{\tilde{x}},$ $x$ replaced by $\tilde{x},$
$\rho=\tilde{r}_{m,\alpha},$ and $\kappa=\tan\left(  \frac{\pi}{4}%
-\alpha\right)  /2.$ Then we get a new $C^{1}$ function $\tilde{\tilde{u}%
}\left(  \tilde{\tilde{x}}\right)  ,$ defined on an open neighborhood of
$\tilde{\tilde{x}}=0.$ By Property 3.3, the three eigen-angles $\tilde
{\tilde{\theta}}_{i}=\arctan\tilde{\tilde{\lambda}}_{i}$ of $D^{2}%
\tilde{\tilde{u}}\left(  \tilde{\tilde{x}}\right)  $ away from the origin
satisfy
\begin{equation}
\left\{
\begin{array}
[c]{l}%
\tilde{\tilde{\theta}}_{1}=\tilde{\theta}_{1}-\frac{\pi}{2}=-\frac{\pi}%
{4}-\alpha+o(1)\\
\tilde{\tilde{\theta}}_{2}=\tilde{\theta}_{2}-\frac{\pi}{2}=-\frac{\pi}%
{4}-\alpha+o(1)\\
\tilde{\tilde{\theta}}_{3}=\tilde{\theta}_{3}-\frac{\pi}{2}+\pi\\
\ \ =\frac{\pi}{2}-\frac{\delta_{m,\alpha}\left(  x\right)  }{\tan\left(
\frac{\pi}{4}-\alpha\right)  }\left\vert D\tilde{\tilde{u}}\left(
\tilde{\tilde{x}}\right)  \right\vert ^{2m-2}\left[  1+O\left(  \left\vert
D\tilde{\tilde{u}}\left(  \tilde{\tilde{x}}\right)  \right\vert ^{m-1}\right)
\right]
\end{array}
\right.  , \label{eigenvalue half pi}%
\end{equation}
where the positive number $\delta_{m,\alpha}\left(  x\right)  $ is bounded
from both below and above uniformly with respect to $\tilde{\tilde{x}},$ and
\[
\sum_{i=1}^{3}\tilde{\tilde{\theta}}_{i}=-2\alpha.
\]

We verify $\tilde{\tilde{u}}$ is still a viscosity solution to (\ref{EsLag})
with $\Theta=-2\alpha$ across the origin. For any quadratic $Q$ touching
$\tilde{\tilde{u}}$ at the origin from below, we have under the
\textquotedblleft diagonalized\textquotedblright\ coordinate system for
$D^{2}\tilde{\tilde{u}}\left(  0\right)  $
\[
D^{2}Q\leq\left[
\begin{array}
[c]{ccc}%
\tan\left(  -\frac{\pi}{4}-\alpha\right)  &  & \\
& \tan\left(  -\frac{\pi}{4}-\alpha\right)  & \\
&  & \infty
\end{array}
\right]  .
\]
It follows that the eigenvalues $\lambda_{i}^{\ast}$ of $D^{2}Q$ must satisfy
\[
\arctan\lambda_{1}^{\ast}\leq-\frac{\pi}{4}-\alpha,\ \arctan\lambda_{2}^{\ast
}\leq-\frac{\pi}{4}-\alpha,\ \text{and }\arctan\lambda_{3}^{\ast}<\frac{\pi
}{2}.
\]
Then the quadratic satisfies
\[
\sum_{i=1}^{3}\arctan\lambda_{i}^{\ast}<-2\alpha.
\]
Observe that we can never arrange any quadratic touching $\tilde{\tilde{u}}$
from above at the origin. Then there is nothing to check. When those testing
quadratics touch the smooth $\tilde{\tilde{u}}$ away from the origin, the
verification according to the definition of viscosity solutions is
straightforward. Thus $\tilde{\tilde{u}}$ is a viscosity solution to
(\ref{EsLag}) with $\Theta=-2\alpha$ in a neighborhood of the origin.

Step 4. Lastly we verify that the solution $\tilde{\tilde{u}}$ is in fact
$C^{1,1/(2m-1)}$ but not $C^{1,\delta}$ for any $\delta>1/\left(  2m-1\right)
$ in a neighborhood of the origin. The latter is easy. From Property 2.3 and
(\ref{distanceexpansion}), we see that%
\[
\left(  0,0,\tilde{x}_{3},D\tilde{u}\left(  0,0,\tilde{x}_{3}\right)  \right)
=\left(  0,0,\tilde{x}_{3},\left[  \tilde{x}_{3}^{2m}\right]  ,\left[
\tilde{x}_{3}^{2m}\right]  ,-2m\varepsilon\tilde{x}_{3}^{2m-1}+\left[
\tilde{x}_{3}^{2m}\right]  \right)  .
\]
It follows that%
\[
\frac{\left\vert \tilde{x}_{3}-0\right\vert }{\left\vert D\tilde{u}\left(
0,0,\tilde{x}_{3}\right)  -D\tilde{u}\left(  0\right)  \right\vert ^{\delta}%
}=\frac{\left\vert \tilde{x}_{3}\right\vert }{\left(  2m\varepsilon+\left[
\tilde{x}_{3}\right]  \right)  ^{\delta}\left\vert \tilde{x}_{3}\right\vert
^{\left(  2m-1\right)  \delta}}\rightarrow\infty\ \ \
\]
as $\tilde{x}_{3}\rightarrow0\ \ $for any $\delta>\frac{1}{2m-1}.$ This shows
that $\tilde{\tilde{u}}$ is not $C^{1,\delta}$ for any $\delta>1/\left(
2m-1\right)  .$

Next we prove that $\tilde{\tilde{u}}$ is $C^{1,1/\left(  2m-1\right)  }$ by
the argument in [NV]. Observe that for $i=1,2,3$%
\begin{align*}
&  \left[  \frac{\left\vert \tilde{\tilde{u}}_{i}\left(  \tilde{\tilde{x}%
}\right)  -\tilde{\tilde{u}}_{i}\left(  \tilde{\tilde{x}}^{\ast}\right)
\right\vert }{\left\vert \tilde{\tilde{x}}-\tilde{\tilde{x}}^{\ast}\right\vert
^{1/\left(  2m-1\right)  }}\right]  ^{2m-1}\\
&  =\left[  \frac{\left\vert \tilde{\tilde{u}}_{i}\left(  \tilde{\tilde{x}%
}\right)  -\tilde{\tilde{u}}_{i}\left(  \tilde{\tilde{x}}^{\ast}\right)
\right\vert }{\left\vert \tilde{\tilde{u}}_{i}^{2m-1}\left(  \tilde{\tilde{x}%
}\right)  -\tilde{\tilde{u}}_{i}^{2m-1}\left(  \tilde{\tilde{x}}^{\ast
}\right)  \right\vert ^{1/\left(  2m-1\right)  }}\right]  ^{2m-1}%
\frac{\left\vert \tilde{\tilde{u}}_{i}^{2m-1}\left(  \tilde{\tilde{x}}\right)
-\tilde{\tilde{u}}_{i}^{2m-1}\left(  \tilde{\tilde{x}}^{\ast}\right)
\right\vert }{\left\vert \tilde{\tilde{x}}-\tilde{\tilde{x}}^{\ast}\right\vert
}\\
&  \leq C\left(  m\right)  \left(  2m-1\right)  \sup_{\tilde{\tilde{x}}%
}\left\vert \tilde{\tilde{u}}_{i}\left(  \tilde{\tilde{x}}\right)  \right\vert
^{2m-2}\left\vert D\tilde{\tilde{u}}_{i}\left(  \tilde{\tilde{x}}\right)
\right\vert \\
&  \leq C\left(  m\right)  \sup_{\tilde{\tilde{x}}}\left\vert D\tilde
{\tilde{u}}\left(  \tilde{\tilde{x}}\right)  \right\vert ^{2m-2}\left\vert
D^{2}\tilde{\tilde{u}}\left(  \tilde{\tilde{x}}\right)  \right\vert \\
&  \leq C\left(  m\right)  \sup_{\tilde{\tilde{x}}}\left\vert D\tilde
{\tilde{u}}\left(  \tilde{\tilde{x}}\right)  \right\vert ^{2m-2}\frac
{1}{\left\vert D\tilde{\tilde{u}}\left(  \tilde{\tilde{x}}\right)  \right\vert
^{2m-2}}\\
&  \leq C\left(  m\right)  ,
\end{align*}
where we used the fact that the scalar function $t^{1/\left(  2m-1\right)  }$
is $C^{1/\left(  2m-1\right)  }\left(  \mathbb{R}^{1}\right)  $ for the first
inequality, and (\ref{eigenvalue half pi}) for the third inequality.

Finally by scaling $u^{m}\left(  x\right)  =\tilde{\tilde{u}}\left(  \tau
x\right)  /\tau^{2}$ with valid radius $\tau$ implicitly depending on $m$ and
the $\tilde{r}_{m,\alpha}$ in Step 1 (We need to make this dependence explicit
and then to have a uniform control with respect to $\varepsilon$ on the valid
radius for the solutions $u^{\varepsilon}$ in Section 4. Our guaranteed valid
radius goes to zero as $m$ goes to infinity.), the desired solutions in
Theorem 1.1 with each fixed $\Theta\in(-\frac{\pi}{2},0]$ are achieved. By
symmetry, $-u^{m}$ are the sought solutions with phase $\Theta\in
\lbrack0,\frac{\pi}{2}).$

\section{Rotate to smooth solutions: proof of Theorem 1.2}

In this section, we create the desired family of solutions by another
corresponding families of $U\left(  3\right)  $ rotations in $\mathbb{C}^{3}$
on top those two in Section 3. For any fixed $\Theta\in\lbrack0,\frac{\pi}%
{2}),$ let $4\gamma=\frac{\pi}{2}-\Theta>0.$ We start the construction by
taking small positive numbers $\varepsilon\in\left(  0,\gamma\right)  $ and
solution $u$ with fixed $m$ in Section 1.

Step 1. We take the $U\left(  3\right)  $ rotation in Step 1 of Section 3 with
$\alpha=\frac{\Theta}{2}-\frac{3\varepsilon}{2}.$ The valid radius of the
rotation and the estimates of the Hessian $D^{2}\tilde{u}$ are still valid. To
prepare the final rotations in the last Step of this section, we require the
following estimates of $D^{2}\tilde{u}$ with eigenvalues $\tilde{\lambda}_{i}$
by shrinking the radius for $\tilde{x}$ or $\left\vert x\right\vert \leq
r_{\Theta}:$%
\begin{equation}
\left\{
\begin{array}
[c]{l}%
\tilde{\theta}_{1}^{\varepsilon}=\arctan\tilde{\lambda}_{1}^{\varepsilon
}=\frac{\pi}{4}-\frac{\Theta}{2}+\frac{3\varepsilon}{2}+o\left(  1\right)
\geq\gamma\\
\tilde{\theta}_{2}^{\varepsilon}=\arctan\tilde{\lambda}_{2}^{\varepsilon
}=\frac{\pi}{4}-\frac{\Theta}{2}+\frac{3\varepsilon}{2}+o\left(  1\right)
\geq\gamma\\
\tilde{\theta}_{3}^{\varepsilon}=\arctan\tilde{\lambda}_{3}^{\varepsilon
}=-\frac{\delta_{m,\alpha}\left(  \tilde{x}\right)  }{\tan\left(  \frac{\pi
}{4}-\frac{\Theta}{2}+\frac{3\varepsilon}{2}\right)  }\left\vert \tilde
{x}\right\vert ^{2m-2}\left[  1+O\left(  \left\vert \tilde{x}\right\vert
^{m-1}\right)  \right]
\end{array}
\right.  , \label{eigenPThm12Step1}%
\end{equation}
where again the positive $\delta_{m,\alpha}\left(  \tilde{x}\right)  $ is
bounded from both below and above uniformly with respect to $\tilde{x}$ and
$\varepsilon,$ further the above estimates and $r_{\Theta}$ are both uniform
with respect to $\varepsilon.$

Step 2. Exactly as in Step 3 of Section 3, we apply Proposition 3.1 with
$\rho=r_{\Theta}$ and $\kappa=\tan\left(  \frac{\pi}{4}-\alpha\right)  /2$ to
$\left(  \tilde{x},D\tilde{u}\right)  $ to get the potential$\ \tilde
{\tilde{u}}^{\varepsilon}$ with $\left(  \tilde{\tilde{x}},D\tilde{\tilde{u}%
}^{\varepsilon}\right)  $ for $\tilde{\tilde{x}}\in D\tilde{u}\left(
B_{\frac{1}{2}\kappa^{2}r_{\Theta}}\right)  .$ It follows from
(\ref{eigenPThm12Step1}) and (\ref{eigenvalue half pi}) that%
\begin{equation}
\left\{
\begin{array}
[c]{l}%
\tilde{\tilde{\theta}}_{1}=-\frac{\pi}{4}-\frac{\Theta}{2}+\frac{3\varepsilon
}{2}+o\left(  1\right)  \geq\gamma-\frac{\pi}{2}\\
\tilde{\tilde{\theta}}_{2}=-\frac{\pi}{4}-\frac{\Theta}{2}+\frac{3\varepsilon
}{2}+o\left(  1\right)  \geq\gamma-\frac{\pi}{2}\\
\tilde{\tilde{\theta}}_{3}=\frac{\pi}{2}-\left\vert o\left(  1\right)
\right\vert
\end{array}
\right.  \label{eigenPThm12Step2}%
\end{equation}
for $\left\vert \tilde{\tilde{x}}\right\vert =\left\vert D\tilde{u}\left(
\tilde{x}\right)  \right\vert \leq\tilde{\tilde{r}}_{\Theta}.$ We need to show
this $\tilde{\tilde{r}}_{\Theta}$ and the above $o\left(  1\right)  $ terms
are still uniform with respect to $\varepsilon,$ and also the $\left\vert
o\left(  1\right)  \right\vert $ term for $\tilde{\tilde{\theta}}_{3}$ never
vanishes when the input $\tilde{\tilde{x}}$ does not vanish (actually this
$\left\vert o\left(  1\right)  \right\vert $ can be made explicit enough by
(\ref{mismatchpower})). All these can be seen from the following inequalities
\begin{equation}
\delta_{6}\left(  m\right)  \left\vert \tilde{x}\right\vert \geq\left\vert
\tilde{\tilde{x}}\left(  \tilde{x}\right)  \right\vert =\left\vert D\tilde
{u}\left(  \tilde{x}\right)  \right\vert \geq\delta_{7}\left(  m\right)
\left\vert \tilde{x}\right\vert ^{2m-1}. \label{mismatchpower}%
\end{equation}
Indeed we see the first inequality by recalling (\ref{h-rotation})%
\begin{gather*}
D\tilde{u}\left(  \tilde{x}\left(  x\right)  \right)  =\tilde{y}\left(
x\right)  =\left(  \cos\alpha\ D^{\prime}u\left(  x\right)  -\sin
\alpha\ x^{\prime},u_{3}\left(  x\right)  \right)  ,\\
\left(  D^{\prime}u,u_{3}\right)  \left(  x^{\prime},x_{3}\right)  =Du\left(
x\right)  \in C^{1},
\end{gather*}
and (\ref{distanceshrinking}).\ We have to work a little harder for the second
inequality. Because of (\ref{AsympHess}) and $\alpha\in\left(  -\frac{3}%
{2}\gamma,\frac{\pi}{4}-4\gamma\right)  ,$ the following convexity for
function
\[
u_{x_{3}}^{\prime}\left(  x^{\prime}\right)  =\cos\alpha\ u\left(  x^{\prime
},x_{3}\right)  -\frac{\sin\alpha}{2}\left\vert x^{\prime}\right\vert ^{2}%
\]
is available%
\[
\cos\alpha\ \left[  D^{2}u\left(  x\right)  \right]  ^{\prime}-\sin
\alpha\ \left[
\begin{array}
[c]{cc}%
1 & \\
& 1
\end{array}
\right]  \geq\frac{\cos\alpha-\sin\alpha}{2}\ I>0
\]
for $\left\vert \left(  x^{\prime},x_{3}\right)  \right\vert \leq r_{\Theta},$
where we shrink $r_{\Theta}$ if necessary. Then we get%
\begin{align*}
\left\vert \tilde{y}\left(  x\right)  \right\vert ^{2}  &  =\left\vert \left(
Du_{x_{3}}^{\prime}\left(  x^{\prime}\right)  ,u_{3}\left(  x\right)  \right)
\right\vert ^{2}=\left\vert \left(  Du_{x_{3}}^{\prime}\left(  x^{\prime
}\right)  -bx^{\prime}+bx^{\prime},u_{3}\left(  x\right)  \right)  \right\vert
^{2}\\
&  \geq\left\vert bx^{\prime}\right\vert ^{2}+2\underset{\geq0}{\underbrace
{\left\langle Du_{x_{3}}^{\prime}\left(  x^{\prime}\right)  -bx^{\prime
},bx^{\prime}\right\rangle }}+\left\vert u_{3}\left(  x\right)  \right\vert
^{2}\\
&  \geq b^{2}\left\vert x^{\prime}\right\vert ^{2}+\left\vert u_{3}\left(
x\right)  \right\vert ^{2},
\end{align*}
where we set $b=\left(  \cos\alpha-\sin\alpha\right)  /2$ for simplicity of
notation. In order to bound $\left\vert x^{\prime}\right\vert ^{2}$ from
below, we use Property 2.3 to obtain%
\begin{align*}
\left\vert D^{\prime}u\left(  x^{\prime},x_{3}\right)  \right\vert ^{2}  &
\leq2\left\vert D^{\prime}u\left(  x^{\prime},x_{3}\right)  -D^{\prime
}u\left(  0,x_{3}\right)  \right\vert ^{2}+2\left\vert D^{\prime}u\left(
0,x_{3}\right)  \right\vert ^{2}\\
&  \leq C_{m}\left\vert x^{\prime}\right\vert ^{2}+\left\vert \left[
r^{2m}\right]  \right\vert ^{2},\ \ \text{or}\\
\left\vert x^{\prime}\right\vert ^{2}  &  \geq\frac{1}{C_{m}}\left\vert
D^{\prime}u\left(  x\right)  \right\vert ^{2}-\left\vert \left[
r^{2m}\right]  \right\vert ^{2}.
\end{align*}
Hence%
\[
\left\vert \tilde{y}\left(  x\right)  \right\vert ^{2}\geq\frac{b^{2}}{C_{m}%
}\left\vert Du\left(  x\right)  \right\vert ^{2}-\left\vert \left[
r^{2m}\right]  \right\vert ^{2},
\]
where we assumed the positive $b^{2}/C_{m}\leq1$ without loss of generality.
By virtue of Property 2.4, we get%
\begin{align*}
\left\vert \tilde{y}\left(  x\right)  \right\vert ^{2}  &  \geq\frac{b^{2}%
}{C_{m}}\left\vert r^{2m-1}\right\vert ^{2}-\left\vert \left[  r^{2m}\right]
\right\vert ^{2}\\
&  \geq\frac{\left(  \delta_{7}\left(  m\right)  \right)  ^{2}}{\delta
_{5}\left(  m\right)  }\left(  \left\vert x\right\vert ^{2m-1}\right)  ^{2}%
\end{align*}
for $\left\vert x\right\vert \leq r_{\Theta}$ and small positive $\delta
_{7}\left(  m,\alpha\right)  ,$ where again we shrink $r_{\Theta}$ if
necessary. By (\ref{distanceshrinking}) we arrive at the second inequality of
(\ref{mismatchpower}).

Step 3. We make a final family of $U\left(  3\right)  $ rotations in
$\mathbb{C}^{3}:\tilde{\tilde{\tilde{z}}}=e^{\varepsilon\sqrt{-1}}%
\tilde{\tilde{z}}.$ Again because $U\left(  3\right)  $ rotation preserves the
length and complex structure, $\mathfrak{M}=\left(  \tilde{\tilde{x}}%
,D\tilde{\tilde{u}}^{\varepsilon}\right)  $ for $\left\vert \tilde{\tilde{r}%
}\right\vert \leq\tilde{\tilde{r}}_{\Theta}$ still a smooth special Lagrangian
submanifold with parameterization%
\[
\left\{
\begin{array}
[c]{c}%
\tilde{\tilde{\tilde{x}}}=\tilde{\tilde{x}}\cos\varepsilon+D\tilde{\tilde{u}%
}^{\varepsilon}\left(  \tilde{\tilde{x}}\right)  \sin\varepsilon\\
\tilde{\tilde{\tilde{y}}}=-\tilde{\tilde{x}}\sin\varepsilon+D\tilde{\tilde{u}%
}^{\varepsilon}\left(  \tilde{\tilde{x}}\right)  \cos\varepsilon
\end{array}
\right.  .
\]
We show that $\mathfrak{M}$ is still a \textquotedblleft
gradient\textquotedblright\ graph over $\tilde{\tilde{\tilde{x}}}$ space. From
(\ref{eigenPThm12Step2}) we know that the function $\tilde{\tilde{u}%
}^{\varepsilon}\left(  \tilde{\tilde{x}}\right)  +\frac{1}{2}\tan\left(
\frac{\pi}{2}-\gamma\right)  \left\vert \tilde{\tilde{x}}\right\vert ^{2}$ is
convex. We then have%
\begin{gather*}
\left\vert \tilde{\tilde{\tilde{x}}}\left(  \tilde{\tilde{x}}\right)
-\tilde{\tilde{\tilde{x}}}\left(  \tilde{\tilde{x}}^{\ast}\right)  \right\vert
^{2}=\left\vert \left(  \tilde{\tilde{x}}-\tilde{\tilde{x}}^{\ast}\right)
\cos\varepsilon+\left(  D\tilde{\tilde{u}}^{\varepsilon}\left(  \tilde
{\tilde{x}}\right)  -D\tilde{\tilde{u}}^{\varepsilon}\left(  \tilde{\tilde{x}%
}^{\ast}\right)  \right)  \sin\varepsilon\right\vert ^{2}\\
=\left\vert
\begin{array}
[c]{c}%
\left(  \tilde{\tilde{x}}-\tilde{\tilde{x}}^{\ast}\right)  \left[
\cos\varepsilon-\tan\left(  \frac{\pi}{2}-\gamma\right)  \sin\varepsilon
\right]  +\\
+\left[  \left(  D\tilde{\tilde{u}}^{\varepsilon}\left(  \tilde{\tilde{x}%
}\right)  -D\tilde{\tilde{u}}^{\varepsilon}\left(  \tilde{\tilde{x}}^{\ast
}\right)  \right)  +\left(  \tilde{\tilde{x}}-\tilde{\tilde{x}}^{\ast}\right)
\tan\left(  \frac{\pi}{2}-\gamma\right)  \right]  \sin\varepsilon
\end{array}
\right\vert ^{2}\\
\geq\left\{
\begin{array}
[c]{c}%
\left\vert \left(  \tilde{\tilde{x}}-\tilde{\tilde{x}}^{\ast}\right)  \left(
\cos\varepsilon-\tan\left(  \frac{\pi}{2}-\gamma\right)  \sin\varepsilon
\right)  \right\vert ^{2}+\\
+2\left[
\begin{array}
[c]{c}%
\cos\varepsilon-\\
\tan\left(  \frac{\pi}{2}-\gamma\right)  \sin\varepsilon
\end{array}
\right]  \sin\varepsilon\underset{\geq0}{\underbrace{\left\langle
\tilde{\tilde{x}}-\tilde{\tilde{x}}^{\ast},\left[
\begin{array}
[c]{c}%
\left(  D\tilde{\tilde{u}}^{\varepsilon}\left(  \tilde{\tilde{x}}\right)
-D\tilde{\tilde{u}}^{\varepsilon}\left(  \tilde{\tilde{x}}^{\ast}\right)
\right)  +\\
\left(  \tilde{\tilde{x}}-\tilde{\tilde{x}}^{\ast}\right)  \tan\left(
\frac{\pi}{2}-\gamma\right)
\end{array}
\right]  \right\rangle }}%
\end{array}
\right\} \\
=\left\vert \tilde{\tilde{x}}-\tilde{\tilde{x}}^{\ast}\right\vert ^{2}\cos
^{2}\varepsilon\left(  1-\frac{\tan\varepsilon}{\tan\gamma}\right)  ^{2}%
\geq\frac{1}{4}\left\vert \tilde{\tilde{x}}-\tilde{\tilde{x}}^{\ast
}\right\vert ^{2}%
\end{gather*}
provided we take $\varepsilon\in\left(  0,\gamma\right)  $ even smaller. It
follows that the smooth $\mathfrak{M}$ is a special Lagrangian graph $\left(
\tilde{\tilde{\tilde{x}}},D\tilde{\tilde{\tilde{u}}}^{\varepsilon}\left(
\tilde{\tilde{\tilde{x}}}\right)  \right)  $ over a domain containing a ball
of radius $\frac{1}{2}\tilde{\tilde{r}}_{\Theta}$ in $\tilde{\tilde{\tilde{x}%
}}$ space. The eigenvalues $\tilde{\tilde{\tilde{\lambda}}}_{i}^{\varepsilon}$
of the Hessian $D^{2}\tilde{\tilde{\tilde{u}}}^{\varepsilon}$ satisfy%
\begin{equation}
\left\{
\begin{array}
[c]{l}%
\tilde{\tilde{\tilde{\theta}}}_{1}^{\varepsilon}=\arctan\tilde{\tilde
{\tilde{\lambda}}}_{1}^{\varepsilon}=-\frac{\pi}{4}-\frac{\Theta}{2}%
+\frac{3\varepsilon}{2}-\varepsilon+o\left(  1\right) \\
\tilde{\tilde{\tilde{\theta}}}_{2}^{\varepsilon}=\arctan\tilde{\tilde
{\tilde{\lambda}}}_{2}^{\varepsilon}=-\frac{\pi}{4}-\frac{\Theta}{2}%
+\frac{3\varepsilon}{2}-\varepsilon+o\left(  1\right) \\
\tilde{\tilde{\tilde{\theta}}}_{3}^{\varepsilon}=\arctan\tilde{\tilde
{\tilde{\lambda}}}_{3}^{\varepsilon}=\frac{\pi}{2}-\varepsilon-\left\vert
o\left(  1\right)  \right\vert
\end{array}
\right.  . \label{eigenPThm12Step3}%
\end{equation}
It follows that $\tilde{\tilde{\tilde{u}}}^{\varepsilon}$ is smooth and
satisfies
\[
\arctan\tilde{\tilde{\tilde{\lambda}}}_{1}^{\varepsilon}+\arctan\tilde
{\tilde{\tilde{\lambda}}}_{2}^{\varepsilon}+\arctan\tilde{\tilde
{\tilde{\lambda}}}_{3}^{\varepsilon}=-\Theta\ \ \text{in \ }B_{\frac{1}%
{2}\tilde{\tilde{r}}_{\Theta}}.
\]

Finally set%
\[
u^{\varepsilon}\left(  x\right)  =-\frac{\tilde{\tilde{\tilde{u}}%
}^{\varepsilon}\left(  \frac{1}{2}\tilde{\tilde{r}}_{\Theta}\ x\right)
}{\left(  \frac{1}{2}\tilde{\tilde{r}}_{\Theta}\right)  ^{2}}.
\]
Observe that the gradients of the potential functions, or the heights of the
special Lagrangian graphs are kept uniformly bounded with respect to
$\varepsilon$ under the above three families of $U\left(  3\right)  $
rotations. Combined with (\ref{eigenPThm12Step3}), we obtain the desired
family of smooth solutions to (\ref{EsLag}) with $n=3$ and fixed $\Theta
\in\lbrack0,\pi/2).$ By symmetry, $-u^{\varepsilon}$ are the other family of
solutions to (\ref{EsLag}) with $n=3$ and fixed $\Theta\in(-\pi/2,0].$

\section{Minimal surface system: proof of Theorem 1.3}

In this section, we prove Theorem 1.3. Take the singular solutions $u^{m}$
from Theorem 1.1 with $\Theta=0$ and $m=2,3,4,\cdots.$ Let%
\[
U^{m}=Du^{m}.
\]
From Property 2.2 and Proposition 3.1, we see that%
\[
\left\vert DU^{m}\left(  y\right)  \right\vert =\left\vert D^{2}u^{m}\left(
y\right)  \right\vert \approx\frac{1}{\left\vert Du^{m}\left(  y\right)
\right\vert ^{2m-2}}.
\]
Here $``\approx"$ means two quantities are equivalent up to a multiple of
constant depending only on the dimension and $m.$ Then we have%
\begin{gather*}
\int_{B_{1}}\left\vert DU^{m}\left(  y\right)  \right\vert ^{p}dy\approx
\int_{B_{1}}\frac{1}{\left\vert Du^{m}\left(  y\right)  \right\vert ^{\left(
2m-2\right)  p}}dy\\
=\int_{Du^{m}\left(  B_{1}\right)  }\frac{1}{\left\vert x\right\vert ^{\left(
2m-2\right)  p}}\left\vert \det\left[  D^{2}u^{m}\left(  y\right)  \right]
^{-1}\right\vert dx\\
\approx\int_{Du^{m}\left(  B_{1}\right)  }\frac{1}{\left\vert x\right\vert
^{\left(  2m-2\right)  p}}\left\vert x\right\vert ^{2m-2}dx_{1}dx_{2}dx_{3}.
\end{gather*}
It follows that%
\[
U^{m}\in W^{1,p}\left(  B_{1}\right)  \ \text{\ for any }p<\frac{2m+1}%
{2m-2}\ \ \text{but\ \ }U^{m}\notin W^{1,\frac{2m+1}{2m-2}}\left(
B_{1}\right)  .
\]
We next show that $U^{m}$ satisfies (\ref{MSS}) in the integral sense, namely%
\[
\int_{B_{1}}\sum_{i,j=1}^{3}\sqrt{g}g^{ij}\left\langle \partial_{x_{i}%
}U,\partial_{x_{j}}\Phi\right\rangle dx=0
\]
for all $\Phi\in C_{0}^{\infty}\left(  B_{1},\mathbb{R}^{3}\right)  .$ This is
because the integrand is $0$ everywhere except at the origin and we have the
following bound on the integrand near the origin. Diagonalizing $D^{2}u^{m},$
we see that%
\begin{gather*}
\left\vert \sum_{i,j=1}^{3}\sqrt{g}g^{ij}\left\langle \partial_{x_{i}%
}U,\partial_{x_{j}}\Phi\right\rangle \right\vert =\left\vert \sum_{i=1}%
^{3}\sqrt{\left(  1+\lambda_{1}^{2}\right)  \cdots\left(  1+\lambda_{3}%
^{2}\right)  }\frac{\lambda_{i}}{1+\lambda_{i}^{2}}\partial_{i}\Phi
^{i}\right\vert \\
\leq C\left(  3,m\right)  \left\vert D^{2}u^{m}\right\vert \left\vert
D\Phi\right\vert =C\left(  3,m\right)  \left\vert DU^{m}\right\vert \left\vert
D\Phi\right\vert \in L^{1},
\end{gather*}
where we used again the fact (\ref{eigenvalue half pi}) that two of the
eigenvalues of $D^{2}u^{m}$ are bounded. The first part of the Theorem 1.3 is proved.

The second part of Theorem 1.3 is straightforward if we take $U^{\varepsilon
}=Du^{\varepsilon}$ with smooth solutions $u^{\varepsilon}$ in Theorem 1.2 for
any fixed $\Theta\in\left(  -\frac{\pi}{2},\frac{\pi}{2}\right)  .$

\end{document}